\documentclass[a4paper,12pt,leqno,english]{article}
\usepackage[utf8]{inputenc}
\usepackage[T1]{fontenc}
\usepackage{babel}
\usepackage{amsmath}
\usepackage{amssymb}
\usepackage{enumitem}
\usepackage{graphicx}
\usepackage{hyperref}
\usepackage{tcolorbox}
\usepackage{pictex}
\usepackage{todonotes} %

\numberwithin{equation}{section}

\newtheorem{theorem+}           {Theorem}      [section]
\newtheorem{definition+}  [theorem+]  {Definition}
\newtheorem{lemma+}  [theorem+]  {Lemma}
\newtheorem{corollary+}  [theorem+]  {Corollary}
\newtheorem{proposition+}  [theorem+]  {Proposition}
\newtheorem{example+}  [theorem+]  {Example}
\newtheorem{problem+}  [theorem+]  {Problem}
\newtheorem{remark+}  [theorem+]  {Remark}

\newenvironment{theorem}{\begin{theorem+}\sl}{\end{theorem+}\rm}
\newenvironment{theoremtx}[1]{\begin{theorem+}[#1]\sl}{\end{theorem+}\rm}

\newenvironment{lemma}{\begin{lemma+}\sl}{\end{lemma+}\rm}
\newenvironment{corollary}{\begin{corollary+}\sl}{\end{corollary+}\rm}
\newenvironment{proposition}{\begin{proposition+}\sl}{\end{proposition+}\rm}

\newenvironment{proof}{\medbreak\noindent{\bf  Proof:}\rm}{\hfill$\square$\rm}
\newenvironment{prooftx}[1]{\medbreak\noindent{\bf 
    #1:}\rm}{\hfill$\square$\rm} 

\title{{\Large \bf 
Polynomials with exponents in  compact convex sets  and 
associated weighted extremal functions -  \\
Generalized product property}}
\author{Bergur Snorrason}
\date{}

\begin{document}

\maketitle

\begin{abstract}
\noindent
A famous result of Siciak is how the Siciak-Zakharyuta functions,
    sometimes called global extremal functions or pluricomplex Green functions with a pole at infinity,
    of two sets relate to the Siciak-Zakharyuta function of their cartesian product.
In this paper Siciak's result is generalized to the setting of Siciak-Zakharyuta functions with growth given by a compact convex set,
    along with discussing why this generalization does not work in the weighted setting.

\end{abstract}

\section{Introduction}
Let $\mathcal{L}(\mathbb{C}^n)$ denote the Lelong class in $\mathbb{C}^n$, consisting of all $u \in \mathcal{PSH}(\mathbb{C}^n)$ such that
    $u(z) \leq \log^+|z| + c_u$, for $z \in \mathbb{C}^n$ and some constant $c_u$.
For every compact $K \subset \mathbb{C}^n$ we define the Siciak-Zakharyuta function of $K$ by
\begin{equation*}
    V_K
    =
    \sup\{u \,; u \in \mathcal{L}(\mathbb{C}^n), u|_K \leq 0\}.
\end{equation*}
Siciak proved in \cite{Sic:1981} a product formula for these functions.
Namely, if $K_j \subset \mathbb{C}^{n_j}$, for $j = 1, \dots, \ell$, $K = K_1 \times \dots \times K_\ell$, and $n = n_1 + \dots + n_\ell$, then
\begin{equation*}
    V_K(z)
    =
    \max\{V_{K_1}(z_1), \dots, V_{K_\ell}(z_\ell)\},
    \quad
    z = (z_1, \dots, z_\ell) \in \mathbb{C}^n, z_j \in \mathbb{C}^{n_j}.
\end{equation*}
See Klimek \cite{Kli:1991}, Theorem $5.1.8$.

Bos and Levenberg in \cite{BosLev:2018} continued the study of pluripotential theory related to convex sets by, among other results,
    generalizing Siciak's product formula.
To state the result we define, for $S \subset \mathbb{R}_+^n$ compact convex and containing $0$, the \emph{logarithmic supporting function} by
\begin{equation*}
    H_S(z)
    =
    \varphi_S(\operatorname{Log}(z))
\end{equation*}
for $z \in \mathbb{C}^{*n}$, where $\varphi_S(\xi) = \sup_{x \in S} \langle x, \xi \rangle$ is the supporting function of $S$
    and $\operatorname{Log}(z) = (\log|z_1|, \dots, \log|z_n|)$.
We then extend the definition of $H_S$ to $\mathbb{C}^n$ by
\begin{equation*}
    H_S(z)
    =
    \varlimsup_{\mathbb{C}^{*n} \ni w \rightarrow z} H_S(w).
\end{equation*}
This allows us to define the Lelong class $\mathcal{L}^S(\mathbb{C}^n)$ with respect to $S$ consisting of $u \in \mathcal{PSH}(\mathbb{C}^n)$
    such that $u(z) \leq H_S(z) + c_u$ for some constant $c_u$.
We also define $\mathcal{L}^S_+(\mathbb{C}^n)$ as the family of those functions $u \in \mathcal{L}^S(\mathbb{C}^n)$
    such that $H_S - c_u \leq u \leq H_S + c_u$.
For every compact $K \subset \mathbb{C}^n$ and function $q$ on $K$ we can then define the weighted Siciak-Zakharyuta function of $K$, $S$, and $q$ by
\begin{equation*}
    V^S_{K, q}
    =
    \sup\{u \,; u \in \mathcal{L}^S(\mathbb{C}^n), u|_K \leq q\}.
\end{equation*}
In the case where $q = 0$ we omit it in the subscript.
The superscript is omitted when $S$ is the standard simplex in $\mathbb{R}^n$,
    that is $S = \Sigma = \operatorname{ch}\{0, e_1, \dots, e_n\}$,
    where $e_1, \dots, e_n$ is the standard basis of $\mathbb{R}^n$ and $\operatorname{ch}$ is denotes the closed convex hull.
This is justified since $V^\Sigma_{K, 0} = V_K$.
The main result of this paper is the following:
\begin{theorem}
\label{mainresult}
Let 
    $K_j \subset \mathbb{C}^{n_j}$ be compact and non-pluripolar for $j = 1, \dots, \ell$, and
    $K = K_1 \times \dots \times K_\ell$,
    $n = n_1 + \dots + n_\ell$,
    $T \subset \mathbb{R}_+^\ell$ be compact convex,
    $S_j \subset \mathbb{R}_+^{n_j}$ be compact convex containing $0$, for $j = 1, \dots, \ell$, and
    $S \subset \mathbb{R}_+^n$ be given by
\begin{equation*}
    S
    =
    \bigcup_{x \in T} (x_1 S_1) \times \dots \times (x_\ell S_\ell).
\end{equation*}
Then
\begin{equation*}
    V^S_K(z)
    =
    \varphi_T(V^{S_1}_{K_1}(z_1), \dots, V^{S_\ell}_{K_\ell}(z_\ell)),
    \quad
    z = (z_1, \dots, z_n), z_j \in \mathbb{C}^{n_j}.
\end{equation*}
\end{theorem}

Taking $T = \operatorname{ch}\{e_1, \dots, e_\ell\}$ and $S_j = \Sigma_j$, where $\Sigma_j$ is the standard simplex in $\mathbb{R}^{n_j}$,
    we get $S = \Sigma_n$ so this generalizes Siciak's original result, since in this case $\varphi_T(\xi) = \max\{\xi_1, \dots, \xi_\ell\}$.
This result also generalizes Theorem $1.1$ in Nguyen and Long \cite{NguTan:2021}.
They prove the following.
\begin{corollary}
Let
    $K_1 \subset \mathbb{C}^{n_1}$ and $K_2 \subset \mathbb{C}^{n_2}$ be compact and non-pluripolar,
    $K = K_1 \times K_2$,
    $S_1 \subset \mathbb{R}_+^{n_1}$ and $S_2 \subset \mathbb{R}_+^{n_2}$ be compact convex and containing
        a neighborhood of $0$ in $\mathbb{R}^{n_1}_+$ and $\mathbb{R}^{n_2}_+$, respectively, and
    $S = \operatorname{ch} ((S_1 \times \{0\}) \cup (\{0\} \times S_2))$.
Then
\begin{equation*}
    V^S_K(z)
    =
    \max\{V^{S_1}_{K_1}(z_1), V^{S_2}_{K_2}(z_2)\}.
\end{equation*}
\end{corollary}

In \cite{NguTan:2021} it is not assumed that $S_1$ and $S_2$ contain a neighborhood of their respective origins, but their proof requires it,
    see Proposition \ref{bg:prop.1} herein.

Another corollary of Theorem \ref{mainresult} is Proposition $2.4$ in Bos and Levenberg \cite{BosLev:2018}.
To state this result we recall the definition of a lower set.
A compact convex $S \subset \mathbb{R}^n_+$, with $0 \in S$, is said to be a \emph{lower set} if for all $s \in S$ the box
    $[0, s_1] \times \dots \times [0, s_n]$ is contained in $S$.
Theorem $5.8$ in \cite{MagSigSigSno:2023} gives several equivalent characterizations for this property.
One of these is that $S$ is a lower set if and only if $\varphi_S(\xi) = \varphi_S(\xi^+)$,
    where $\xi^+ = (\xi_1^+, \dots, \xi_n^+)$ and $\xi_j^+ = \max\{0, \xi\}$, for all $\xi \in \mathbb{R}^n$.
We refer to the smallest lower set containing $S$ as the \emph{lower hull of $S$} denoted by $\hat{S}_\ell$,
    and note that $S$ is a lower set if and only if $S = \hat{S}_\ell$.
The supporting function of $\hat{S}_\ell$ is therefore given by $\varphi_{\hat{S}_\ell}(\xi) = \varphi_S(\xi^+)$.
If we assume $\ell = n$ and $S_j = [0, 1]$ in Theorem \ref{mainresult} we have that
\begin{equation*}
    \varphi_S(\xi)
    =
    \varphi_T(\varphi_{[0, 1]}(\xi_1), \dots, \varphi_{[0, 1]}(\xi_n))
    =
    \varphi_T(\xi_1^+, \dots, \xi_n^+)
    =
    \varphi_{\hat{T}_\ell}(\xi)
\end{equation*}
since $\varphi_{[a, b]}(\xi) = \max\{a\xi, b\xi\}$.
So $S = \hat{T}_\ell$, since supporting functions uniquely determine their sets.
This is clarified further in Section \ref{sec:bg}.
So setting $\ell = n$ and $S_j = [0, 1]$ leads to the following, which is a generalization of Proposition $2.4$ in \cite{BosLev:2018}.
\begin{corollary}
Let
    $K_1, \dots, K_n \subset \mathbb{C}$ be compact and non-polar,
    $K = K_1 \times \dots \times K_n$, and
    $S \subset \mathbb{R}_+^n$ convex compact and containing $0$.
Then
\begin{equation*}
    V^{\hat{S}_\ell}_K(z)
    =
    \varphi_S(V_{K_1}(z_1), \dots, V_{K_n}(z_n)).
\end{equation*}
\end{corollary}

In \cite{BosLev:2018} this is proven in the setting where $S$ is a lower set.
Then the formula becomes
\begin{equation}
\label{intro:eq.1}
    V^S_K(z)
    =
    \varphi_S(V_{K_1}(z_1), \dots, V_{K_n}(z_n)).
\end{equation}
Levenberg and Perera \cite{LevPer:2020} claim to prove that the formula also holds if we only assume that $a \Sigma \subset S$ for some $a > 0$, where
    $\Sigma$ is the standard simplex in $\mathbb{R}^n$.
Subsequently Nguyen and Long \cite{NguTan:2021} claimed it holds under the relaxed condition that $S$ is a convex body,
    that is when the interior of $S$ is not empty.
These results can not hold.
Both make the erroneous assumption that the right hand side of (\ref{intro:eq.1}) is in $\mathcal{L}^S(\mathbb{C}^n)$,
    but Theorem \ref{mainresult} tells us that it is in $\mathcal{L}^{\hat{S}_\ell}(\mathbb{C}^n)$.
We can also show that these result are wrong by an explicit counterexample.

Let $K_1 = K_2 = \overline{\mathbb{D}}$ and $K = K_1 \times K_2 = \overline{\mathbb{D}}^2$.
By Proposition $4.3$ in \cite{MagSigSigSno:2023}, we have that $V^S_K = H_S$, for every $0 \in S \subset \mathbb{R}_+^n$ compact and convex.
If we set $S = \operatorname{ch}\{(0, 0), (1, 0), (1, 1), (0, a)\}$ then $\varphi_S(\xi) = \max\{\xi_1^+, (\xi_1 + \xi_2)^+, a\xi_2^+\}$,
    for all $\xi \in \mathbb{R}^2$, and thus
\begin{equation*}
    H_S(z)
    =
    \max\{\log^+|z_1|, (\log|z_1| + \log|z_2|)^+, a\log^+|z_2|\},
\end{equation*}
for $z \in \mathbb{C}^2$.
But
\begin{equation*}
    \varphi_S(V_{\overline{\mathbb{D}}}(z_1), V_{\overline{\mathbb{D}}}(z_2))
    =
    \max\{\log^+|z_1|, \log^+|z_1| + \log^+|z_2|, a\log^+|z_2|\},
\end{equation*}
for $z \in \mathbb{C}^2$, and, for $\zeta \in \mathbb{C}$ with $|\zeta| > 1$,
\begin{equation*}
    V^S_K(\zeta^{-1}, \zeta)
    =
    H_S(\zeta^{-1}, \zeta)
    =
    a \log |\zeta|
    <
    \log |\zeta|
    =
    \varphi_S(V_{\overline{\mathbb{D}}}(\zeta), V_{\overline{\mathbb{D}}}(\zeta^{-1})),
\end{equation*}
when $a < 1$.

We get more corollaries when we have an explicit formula for $\varphi_T$.
One immediate example is taking $T = \{(1, 1)\}$, since then $\varphi_T(\xi) = \xi_1 + \xi_2$ and $S = S_1 \times S_2$.
\begin{corollary}
Let
    $K_1 \subset \mathbb{C}^{n_1}$ and $K_2 \subset \mathbb{C}^{n_2}$ be compact and non-pluripolar, and
    $S_1 \subset \mathbb{R}_+^{n_1}$ and $S_2 \subset \mathbb{R}_+^{n_2}$ be compact convex and containing $0$.
Then
\begin{equation*}
    V^{S_1 \times S_2}_{K_1 \times K_2}(z)
    =
    V^{S_1}_{K_1}(z_1) + V^{S_2}_{K_2}(z_2).
\end{equation*}
\end{corollary}

It is well known that $\varphi_T(\xi) = \|\xi^+\|_q$ if $T = B \cap \mathbb{R}_+^n$,
    where $B$ is the unit ball with respect to the norm $\| \cdot \|_p$ and $p, q \geq 1$ satisfy $1/p + 1/q = 1$,
    see the discussion of equations ($2.2.12$) in \cite{Hormander:convexity}.
\begin{corollary}
Let $n_1, \dots, n_\ell$ be natural numbers,
    $n = n_1 + \dots + n_\ell$,
    $K_j \subset \mathbb{C}^{n_j}$ be compact and non-pluripolar for $j = 1, \dots, \ell$, and
    $K = K_1 \times \dots \times K_\ell$,
    $S_j \subset \mathbb{R}_+^{n_j}$ be compact convex and containing $0$, for $j = 1, \dots, \ell$, and
    $S \subset \mathbb{R}_+^n$ be given by
\begin{equation*}
    S
    =
    \bigcup_{\substack{x \in \mathbb{R}_+^\ell \\ \|x\|_p \leq 1}}
        (x_1 S_1) \times \dots \times (x_\ell S_\ell).
\end{equation*}
Then
\begin{equation*}
    V^S_K(z)^q
    =
    \|(V^{S_1}_{K_1}(z_1), \dots, V^{S_\ell}_{K_\ell}(z_\ell))\|_q^q
    =
    V^{S_1}_{K_1}(z_1)^q + \dots + V^{S_\ell}_{K_\ell}(z_\ell)^q.
\end{equation*}
\end{corollary}

Note that $B \cap \mathbb{R}_+^n$ is a lower set so the previous result becomes particularly explicit when $\ell = n$.
\begin{corollary}
Let
    $p, q > 1$ with $1/p + 1/q = 1$,
    $K_j \subset \mathbb{C}$ be compact and non-pluripolar for $j = 1, \dots, n$, and
    $K = K_1 \times \dots \times K_n$,
    $S \subset \mathbb{R}_+^n$ is given by $S = \{x \in \mathbb{R}_+^n\,; \|x\|_p \leq 1\}$.
Then
\begin{equation*}
    V^S_K(z)^q
    =
    V_{K_1}(z_1)^q + \dots + V_{K_n}(z_n)^q.
\end{equation*}
\end{corollary}

A natural question is if it is possible to generalize Theorem \ref{mainresult} to the weighted case.
The answer turns out to be negative, as is shown in Propositions \ref{weight:prop.1} and \ref{weight:prop.2}.
As a follow up, we will look into when the sublevel sets of $V^S_K$ are not convex, even if $K$ is convex.

For general results on weighted Siciak-Zakharyuta functions and 
their properties see
\cite{MagSigSigSno:2023} and the references therein.  See also 
\cite{MagSigSig:2023} and \cite{MagSigSno:2023}.

\subsection*{Acknowledgment}  
The results of this paper are a part of a research project, 
{\it Holomorphic Approximations and Pluripotential Theory},
with  project grant 
no.~207236-051 supported by the Icelandic Research Fund.
I would like to thank the Fund for its support and the Mathematics Division, Science Institute, University of Iceland, for hosting the project.
I thank my supervisors Ragnar Sigurðsson and Benedikt Steinar Magnússon for their guidance and careful reading the paper.

\section{Background}
\label{sec:bg}
This section is an overview of required fundamental results from \cite{MagSigSigSno:2023}.
We will use the notation $\mathbb{C}^* = \mathbb{C} \setminus \{0\}$ and $\mathbb{C}^{*n} = (\mathbb{C}^*)^n$.
We fix a compact convex $S \subset \mathbb{R}_+^n$ such that $0 \in S$.
Recall that we define the logarithmic supporting function of $S$ by
\begin{equation*}
    H_S(z)
    =
    \left \{
        \begin{array}{l l}
            \varphi_S(\operatorname{Log}(z)), & z \in \mathbb{C}^{*n}\\
            \varlimsup\limits_{\mathbb{C}^{*n} \ni w \rightarrow z} H_S(w), & z \not \in \mathbb{C}^{*n}\\
        \end{array}
    \right .
\end{equation*}
where $\varphi_S(\xi) = \sup_{x \in S} \langle x, \xi \rangle$ and
    $\operatorname{Log}(z) = (\log|z_1|, \dots, \log|z_n|)$.
We then define the Lelong class $\mathcal{L}^S(\mathbb{C}^n)$
    with respect to $S$ as those $u \in \mathcal{PSH}(\mathbb{C}^n)$ such that $u \leq H_S + c_u$, for some constant $c_u$.
For a function $q$ on a compact $K \subset \mathbb{C}^n$ we define the Siciak-Zakharyuta function of $K$, $q$, and $S$ by
\begin{equation*}
    V^S_{K, q}
    =
    \sup
    \{
        u \,; u \in \mathcal{L}^S(\mathbb{C}^n), u|_K \leq q
    \}.
\end{equation*}
The function $q$ is referred to as a \emph{weight} and is said to be \emph{admissible} if it is lower semicontinuous and the set
    $\{z \in K \,; q(z) < +\infty\}$ is non-pluripolar.
Let now $K \subset \mathbb{C}^n$ be compact and $q$ an admissible weight on $K$.
By Proposition $5.4$ in \cite{MagSigSigSno:2023} we have that $V^S_{K, q}$ is lower semicontinuous on $\mathbb{C}^{*n}$ and if furthermore
    $V^{S*}_{K, q} \leq q$ on $K$ then $V^{S*}_{K, q} = V^S_{K, q} \in \mathcal{L}^S(\mathbb{C}^n)$,
    and consequently, $V^S_{K, q}$ is continuous on $\mathbb{C}^{*n}$.
The assumption that $V^{S*}_{K, q} \leq q$ is not restrictive, since $V^{S*}_{K + \varepsilon \overline{\mathbb{D}}^n, q'} \leq q'$,
    where $q'$ is a continuous (and thus admissible) weight on $K + \varepsilon \overline{\mathbb{D}}^n$, and
    $V^S_{K_j, q_j} \nearrow V^S_{K, q}$ if $K_j \searrow K$ and $q_j \nearrow q$.
See Lemma $5.2$, and Propositions $5.3$ and $4.8$ in \cite{MagSigSigSno:2023}.
The Siciak-Zakharyuta functions are also continuous under decreasing sequences in $S$ under some conditions.
Namely, if $T_j \searrow S$, such that $V^{T_j*}_{K, q} \leq q$ for some $j \in \mathbb{N}$, then $V^{T_j}_{K, q} \searrow V^S_{K, q}$.
See Proposition $4.8$ in \cite{MagSigSigSno:2023}.

Fundamental to this study is that we can explicitly determine some Siciak-Zakharyuta functions.
By Proposition $4.3$ in \cite{MagSigSigSno:2023} we know that $V^S_K = H_S$ if
\begin{equation*}
    \mathbb{T}^n \subset K \subset \{z \in \mathbb{C}^n \,; H_S(z) = 0\},
\end{equation*}
where $\mathbb{T}$ is the unit circle in $\mathbb{C}$.
An example of such a $K$ is $\overline{\mathbb{D}}^n$.

Lemma $2.2$ in Nguyen and Long \cite{NguTan:2021} contains a comparison result that can be used, in certain cases,
    to characterize the Siciak-Zakharyuta functions in terms of maximality.
For the convenience of the reader we include it here, along with their proof.
\begin{lemma}
    \label{bg:lemma.1}
Let $u, v \in \mathcal{PSH}(\mathbb{C}^n)$ such that:
\begin{enumerate}
    \item[{\bf (i)}  ] $u \leq v$ on $K$,
    \item[{\bf (ii)} ] $\varliminf\limits_{|z| \rightarrow \infty} v(z) = +\infty$,
    \item[{\bf (iii)}] $\sup_{z \in \mathbb{C}^n} u(z) - v(z) < +\infty$,
    \item[{\bf (iv)} ] $v$ is maximal on $\mathbb{C}^n \setminus K$,
    \item[{\bf (v)}  ] $v$ is bounded from below on $K$,
\end{enumerate}
then $u \leq v$ on $\mathbb{C}^n$.
\end{lemma}
\begin{proof}
By {\bf(v)} we may assume that $v \geq 0$ on $K$.
We now fix $\lambda > 1$ and note that, by {\bf(iii)}, there exists a constant $C$ such that $u \leq v + C$.
By {\bf (ii)} and the upper semicontinuity of $v$ we can take $R > 0$ such that $v(z) > C(\lambda - 1)^{-1}$,
    for $z \in U_R = \{z \in \mathbb{C}^n \,; |z| \geq R\}$.
We then get
\begin{equation*}
    u(z)
    \leq
    v(z) + C
    =
    \lambda v(z) + (1 - \lambda) v(z) + C
    \leq
    \lambda v(z)
\end{equation*}
for $z \in U_R$. By the positivity of $v$ on $K$ and {\bf (i)} we have that $u \leq \lambda v$ on $K \cup U_R$.
Note that, by {\bf (iv)}, $v$ is maximal on $\mathbb{C}^n \setminus (K \cup U_R)$ so $u \leq \lambda v$ on $\mathbb{C}^n$.
This holds for all $\lambda > 1$ so we conclude that $u \leq v$.
\end{proof}
\begin{proposition}
\label{bg:prop.1}
Let $S \subset \mathbb{R}_+^n$ be compact and convex with $0 \in S$ such that $S$ contains a neighborhood of $0$ in $\mathbb{R}^n_+$,
    $K \subset \mathbb{C}^n$ compact,
    $q$ an admissible weight on $K$, and
    $v \in \mathcal{L}^S_+(\mathbb{C}^n)$.
    If $V^{S*}_{K, q} \leq v \leq q$ on $K$ and $v$ is maximal on $\mathbb{C}^n \setminus K$ then $v = V^{S*}_{K, q}$ on $\mathbb{C}^n$.
\end{proposition}
\begin{proof}
Since $v \in \mathcal{L}^S_+(\mathbb{C}^n)$ and $v \leq q$ on $K$ it is clear that $v \leq V^{S*}_{K, q}$.
By assumption $S$ contains a neighborhood of $0$ in $\mathbb{R}^n_+$, that is there exists $a > 0$ such that $a\Sigma \subset S$, so
\begin{equation*}
    v(z)
    \geq
    H_S(z) + C
    \geq
    a \log^+ \|z\|_\infty + C,
\end{equation*}
so $\displaystyle\varliminf_{|z| \rightarrow \infty} v(z) = +\infty$ and $v$ is bounded below on $K$.
So $v$ and $u = V^{S*}_{K, q}$ satisfy the conditions of Lemma \ref{bg:lemma.1} and thus $V^{S*}_{K, q} \leq v$.
\end{proof}

In \cite{NguTan:2021} Nguyen and Long do not include the assumption that $S$ contains a neighborhood of $0$ in $\mathbb{R}^n_+$.
So the proof of their Theorem $1.1$ is incomplete.

Central to the proof of our main result is the Siciak-Zakharyuta theorem.
The version we require can be found in \cite{MagSigSig:2023}, as Theorem $1.1$, and is restated here for the convenience of the reader.
The Siciak-Zakharyuta theorem relates $V^S_{K, q}$ to an extremal function given by polynomials.
This function, the Siciak extremal function, is defined by
\begin{equation*}
    \Phi^S_{K, q}(z)
    =
    \varlimsup_{m \rightarrow \infty} \sup\{|p(z)|^{1/m}\,; p \in \mathcal{P}^S_m(\mathbb{C}^n), p e^{-mq}|_K \leq 1\},
\end{equation*}
where $\mathcal{P}^S_m(\mathbb{C}^n)$ consists of all polynomials $p$ of the form
    $p(z) = \sum\limits_{\alpha \in (mS) \cap \mathbb{N}^n} c_\alpha z^\alpha$.
\begin{theoremtx}{\cite{MagSigSig:2023}, Theorem $1.1$}
\label{bg:thm.2}
Let
    $S \subset \mathbb{R}_+^n$ be compact convex and containing $0$,
    $K \subset \mathbb{C}^n$ be compact, and
    $q$ an admissible weight on $K$.
Then
\begin{equation*}
    V^S_{K, q}(z)
    =
    \log \Phi^S_{K, q}(z),
\end{equation*}
for all $z \in \mathbb{C}^{*n}$, if and only if $S = \overline{S \cap \mathbb{Q}^n}$.
\end{theoremtx}

Let us now turn our attention to the generalization of the product property.
Recall that if $A \subset \mathbb{R}^n$ is compact and convex then its supporting function
    $\varphi_A(\xi) = \sup_{x \in A} \langle x, \xi \rangle$ is positively homogeneous and convex,
    that is $\varphi_A(t \xi) = t \varphi_A(\xi)$ and $\varphi_A(\xi + \eta) \leq \varphi_A(\xi) + \varphi_A(\eta)$, for
    $t \in \mathbb{R}_+$ and $\eta, \xi \in \mathbb{R}^n$.
Furthermore, if $\varphi \colon \mathbb{R}^n \rightarrow \mathbb{R}$ is positively homogeneous and convex
    then it is the supporting function of precisely one compact convex set.
See Theorem $2.2.8$ in \cite{Hormander:convexity}.
Let $0 \in S_j \subset \mathbb{R}^{n_j}$ be compact convex for $j = 1, \dots, \ell$ and
    $T \subset \mathbb{R}_+^\ell$ compact convex.
Note that we do not assume that $0 \in T$.
By this assumption $\varphi_T$ is increasing in each variable, so if $u_j \in \mathcal{L}^{S_j}(\mathbb{C}^n)$ for $j = 1, \dots, \ell$, then
\begin{equation*}
    u(z)
    =
    \varphi_T(u_1(z_1), \dots, u_\ell(z_\ell))
    \leq
    \varphi_T(H_{S_1}(z_1), \dots, H_{S_\ell}(z_\ell)) + \varphi_T(c),
\end{equation*}
for $c \in \mathbb{R}^\ell$.
The convexity of $\varphi_T$ implies that $u \in \mathcal{PSH}(\mathbb{C}^n)$.
This leads us to define $\varphi \colon \mathbb{R}^n \rightarrow \mathbb{R}$ by
\begin{equation*}
    \varphi(\xi)
    =
    \varphi_T(\varphi_{S_1}(\xi_1), \dots, \varphi_{S_\ell}(\xi_\ell)),
\end{equation*}
where $\xi_j \in \mathbb{R}^{n_j}$, for $j = 1, \dots, \ell$.
Note that $\varphi$ is positively homogeneous and convex so it is the supporting function of some compact convex $S \subset \mathbb{R}^n$.
We have thus shown that $u \in \mathcal{L}^S(\mathbb{C}^n)$ and, since $\varphi_T(0) = 0$,
\begin{equation}
\label{bg:eq.1}
    \varphi_T(V^{S_1}_{K_1}(z_1), \dots, V^{S_\ell}_{K_\ell}(z_\ell))
    \leq
    V^S_K(z),
\end{equation}
for $z \in \mathbb{C}^n$, where $K = K_1 \times \dots \times K_\ell$.

To determine an explicit formula for $S$ we set
\begin{equation*}
    \widetilde{S}_j
    =
    \{0_1\} \times \dots \times \{0_{j - 1}\} \times S_j \times \{0_{j + 1}\} \times \dots \times \{0_\ell\},
\end{equation*}
where $0_j$ is the origin of $\mathbb{R}^{n_j}$.
For $\xi = (\xi_1, \dots, \xi_\ell) \in \mathbb{R}^n$ we have $\varphi_{\widetilde{S}_j}(\xi) = \varphi_{S_j}(\xi_j)$, for $j = 1, \dots, \ell$, so
$
    \varphi_S
    =
    \varphi_T(\varphi_{\widetilde{S}_1}, \dots, \varphi_{\widetilde{S}_\ell}).
    $
For compact convex sets $A$ and $B$ and $a > 0$, we have that $a\varphi_A + \varphi_B = \varphi_{aA + B}$, and
    if $(A_\alpha)_{\alpha \in I}$, $I \neq \emptyset$, is a family of compact sets, $A = \operatorname{ch} \cup_{\alpha \in I} A_\alpha$, and
\begin{equation*}
    \varphi(\xi)
    =
    \sup_{\alpha \in I}
        \varphi_{A_\alpha}(\xi)
\end{equation*}
is bounded for every $\xi$, then $\varphi$ is the supporting function of $A$.
We therefore have
\begin{equation*}
    \varphi_S(\xi)
    =
    \sup_{x \in T}
    (
        x_1 \varphi_{\widetilde{S}_1}(\xi)
        + \dots +
        x_\ell \varphi_{\widetilde{S}_\ell}(\xi)
    )
    =
    \sup_{x \in T}
    (
        \varphi_{x_1 \widetilde{S}_1 + \dots + x_\ell \widetilde{S}_\ell}(\xi)
    )
\end{equation*}
so $S = \operatorname{ch}\cup_{x \in T} (x_1 \widetilde{S}_1 + \dots + x_\ell \widetilde{S}_\ell)$.
Actually, the union is convex:

\begin{lemma}
Let $A\subset \mathbb{R}^\ell_+$ and $B_1,\dots,B_\ell\subset \mathbb{R}^n$ be convex sets.
Then
\begin{equation*}
    C
    =
    \bigcup_{a \in A} a_1 B_1 + \dots + a_\ell B_\ell
\end{equation*}
is a convex subset of $\mathbb{R}^n$.
If $A, B_1, \dots, B_\ell$ are compact, then $C$ is compact.
\end{lemma}

\begin{proof}  Let $c_1 = x_1 w_1 + \dots + x_\ell w_\ell,
  c_2 = y_1 z_1 + \dots + y_\ell z_\ell \in C$ and $t \in [0,1]$,
where $x = (x_1, \dots, x_\ell)$, $y = (y_1, \dots, y_\ell) \in A$ and $w_j, z_j \in B_j$
for $j = 1, \dots, \ell$.
Then
\begin{equation*}
    (1 - t)c_1 + tc_2
    =
    (1 - t)x_1 w_1 + ty_1 z_1 + \dots + (1 - t)x_\ell w_\ell + ty_\ell z_\ell.
\end{equation*}
Since $B_j$ is convex we have $(1 - t)x_j w_j + ty_j z_j \in ((1 - t)x_j + ty_j)B_j$, so for $a = (1 - t)x + ty$ we have
$(1 - t)c_1 + tc_2 \in a_1B_1 + \dots + a_\ell B_\ell \subset  C$.
The last statement follows by a simple sequence argument.
\end{proof}

Finally note that $\varphi_S = \varphi_T(\varphi_{S_1}, \dots, \varphi_{S_\ell}) \geq 0$.
It is well known that the supporting function of a compact convex set is positive if and only if the set contains the origin.

So, in conclusion, $S \subset \mathbb{R}_+^n$ is compact and convex, contains $0$, is given by
\begin{equation}
\label{eq:2.7}
    S
    =
    \bigcup_{x \in T} x_1 \widetilde{S}_1 + \dots + x_\ell \widetilde{S}_\ell
    =
    \bigcup_{x \in T} (x_1 S_1) \times \dots \times (x_\ell S_\ell),
\end{equation}
has the supporting function
\begin{equation}
\label{bg:eq.15}
    \varphi_S(\xi)
    =
    \varphi_T(\varphi_{S_1}(\xi_1), \dots, \varphi_{S_\ell}(\xi_\ell)),
    \quad
    \xi = (\xi_1, \dots, \xi_\ell), \xi_j \in \mathbb{R}^{n_j},
\end{equation}
and the logarithmic supporting function
\begin{equation*}
    H_S(z)
    =
    \varphi_T(H_{S_1}(z_1), \dots, H_{S_\ell}(z_\ell)),
    \quad
    z = (z_1, \dots, z_\ell), z_j \in \mathbb{C}^{n_j}.
\end{equation*}

\section{Proof of the main result}
We will prove Theorem \ref{mainresult} by applying the Siciak-Zakharyuta theorem along with a product variant of the Bernstein-Walsh inequality.
Recall that the Bernstein-Walsh inequality states that
\begin{equation*}
    |f(z)|
    \leq
    \|f\|_{L^\infty(K)} e^{mV^S_K(z)},
\end{equation*}
for
    $z \in \mathbb{C}^n$,
    $0 \in S \subset \mathbb{R}^n_+$ compact and convex,
    $K \subset \mathbb{C}^n$ non-pluripolar compact, and
    $f \in \mathcal{P}^S_m(\mathbb{C}^n)$.
This is a consequence of Theorem $3.6$ in \cite{MagSigSigSno:2023},
    namely that $f$, holomorphic on $\mathbb{C}^n$, is in $\mathcal{P}^S_m(\mathbb{C}^n)$
        if and only if $\log |f|^{1/m} \in \mathcal{L}^S(\mathbb{C}^n)$.

The proof of the product variant of the Bernstein-Walsh inequality follows Klimek's proof of Theorem $5.1.8$ in \cite{Kli:1991}
    and Siciak's proof of Proposition $5.9$ in \cite{Sic:1981},
    in a similar way as the proof of Bos and Levenberg \cite{BosLev:2018}.
\begin{proposition}
\label{main:prop.1}
Let 
    $K_j \subset \mathbb{C}^{n_j}$ be compact and non-pluripolar for $j = 1, \dots, \ell$, and
    $K = K_1 \times \dots \times K_\ell$,
    $n = n_1 + \dots + n_\ell$,
    $T \subset \mathbb{R}_+^\ell$ be compact convex,
    $S_j \subset \mathbb{R}_+^{n_j}$ be compact convex and containing $0$, for $j = 1, \dots, \ell$,
    $S \subset \mathbb{R}_+^n$ be given by (\ref{eq:2.7}), and
    $f \in \mathcal{P}^S_m(\mathbb{C}^n)$.
Then
\begin{equation*}
    |f(z)|
    \leq
    \|f\|_{L^\infty(K)}
        e^{m\varphi_T(V^{S_1}_{K_1}(z_1), \dots, V^{S_\ell}_{K_\ell}(z_\ell))}
\end{equation*}
for $z \in \mathbb{C}^n$.
\end{proposition}
\begin{proof}
We will assume that $\overline{S_j \cap \mathbb{Q}^{n_j}} \neq \{0\}$ for $j = 1, \dots, \ell$.
If this were not the case then $f(z)$, $z \in \mathbb{C}^n$, would not depend on $z_j \in \mathbb{C}^{n_j}$.
We fix an $r > 0$ and let $G_j = K_j + r\mathbb{D}^{n_j}$ for $j = 1, \dots, \ell$ and $\Omega_r = G_1 \times \dots \times G_\ell$.

We now need to define an orthonormal basis for $\mathcal{P}^{S_j}(\mathbb{C}^{n_j})$ for $j = 1, \dots, \ell$, as subspaces of $L^2(G_j)$.
To do this we define $\rho_j \colon \mathbb{R}_+S_j \cap \mathbb{N}^{n_j} \rightarrow \mathbb{R}_+$ by
\begin{equation*}
    \rho_j(\alpha_j)
    =
    \inf \{t \in \mathbb{R}\,; \alpha_j \in tS_j\}
\end{equation*}
    and let $\kappa_j \colon \mathbb{N} \rightarrow \mathbb{R}_+S_j \cap \mathbb{N}^{n_j}$ be a bijection such that
$\rho_j(\kappa_j(k)) \leq \rho_j(\kappa_j(k + 1))$ for all $k \in \mathbb{N}$.
We then set $e_k = z^{\kappa_j(k)}$,
    apply the Gram-Schmidt process to $(e_k)_{k \in \mathbb{N}}$ to get $(\hat{e}_k)_{k \in \mathbb{N}}$, and
    define $p_{j, \alpha_j} = \hat{e}_{\kappa_j^{-1}(\alpha_j)}$, for all $\alpha_j \in \mathbb{R}_+S_j \cap \mathbb{N}^{n_j}$.
    This construction implies that if $z_j^{\alpha_j} = \sum_k c_{j, k} p_{j, \beta_k}(z_j)$
    then $p_{j, \beta_k} \in \mathcal{P}^{\rho_j(\alpha_j)S_j}_1(\mathbb{C}^{n_j})$ for all $k \in \mathbb{N}$ such that $c_{j, k} \neq 0$ and
    $z_j \in \mathbb{C}^{n_j}$.

Now we define $p_\alpha = p_{1, \alpha_1} \dots p_{\ell, \alpha_\ell}$
    for $\alpha = (\alpha_1, \dots, \alpha_\ell) \in \mathbb{R}_+S \cap \mathbb{N}^n$,
    where $\alpha_j \in \mathbb{N}^{n_j}$ for $j = 1, \dots, \ell$.
We now need to show that $\{p_\alpha\,; \alpha \in \mathbb{R}_+S \cap \mathbb{N}^n\}$ is a basis for $\mathcal{P}^S(\mathbb{C}^n)$.
First we show that $p_\alpha \in \mathcal{P}^S(\mathbb{C}^n)$ for all $\alpha \in \mathbb{R}_+S \cap \mathbb{N}^n$,
    and then we show that the span of $\{p_\alpha\,; \alpha \in \mathbb{R}_+S \cap \mathbb{N}^n\}$ is $\mathcal{P}^S(\mathbb{C}^n)$.

Let $m \in \mathbb{N}$ and $\alpha = (\alpha_1, \dots, \alpha_\ell) \in mS \cap \mathbb{N}^n$,
    where $\alpha_j \in \mathbb{R}_+^{n_j}$ for $j = 1, \dots, \ell$.
By the definition of $S$ there exists $x \in T$ such that $\alpha_j \in m x_j S_j$ for $j = 1, \dots, \ell$.
By construction $p_{j, \alpha_j} \in \mathcal{P}^{x_j S_j}_m(\mathbb{C}^{n_j})$ for $j = 1, \dots, \ell$,
    so for some $C > 0$, and all $z \in \mathbb{C}^n$, 
\begin{equation*}
    |p_\alpha(z)|
    =
    |p_{1, \alpha_1}(z_1)| \dots |p_{\ell, \alpha_\ell}(z_\ell)|
    \leq
    C e^{mH_{x_1 S_1}(z_1)} \cdots e^{mH_{x_\ell S_\ell }(z_\ell )} 
\end{equation*}
\begin{equation*}
    =
    C e^{m(x_1 H_{S_1}(z_1) + \dots + x_\ell H_{S_\ell}(z_\ell))} 
    \leq
    C e^{m \varphi_T(H_{S_1}(z_1), \dots, H_{S_\ell}(z_\ell))}
    =
    C e^{m H_S(z)},
\end{equation*}
so we infer $p_\alpha \in \mathcal{P}^S_m(\mathbb{C}^n)$.

To show that $\{p_\alpha\,; \alpha \in \mathbb{R}_+S \cap \mathbb{N}^n\}$ spans $\mathcal{P}^S(\mathbb{C}^n)$ it is sufficient to
    show that $z^\alpha$ belongs to the span for all $\alpha \in \mathbb{R}_+S \cap \mathbb{N}^n$.
We let $m \in \mathbb{N}$ and $\alpha = (\alpha_1, \dots, \alpha_\ell) \in mS \cap \mathbb{N}^n$,
    where $\alpha_j \in \mathbb{R}_+^{n_j}$.
Then there exists $x \in T$ such that $\alpha_j \in m x_j S_j$ for $j = 1, \dots, \ell$.
For $j = 1, \dots, \ell$ we let $\beta_{j, k} \in \mathbb{R}_+S_j \cap \mathbb{N}^{n_j}$ such that
\begin{equation*}
    z_j^{\alpha_j}
    =
    \sum_{k = 1}^{d_j} c_{j, k} p_{j, \beta_{j, k}}(z_j),
\end{equation*}
where $c_{j, k} \in \mathbb{C}^*$, for $k = 1, \dots, d_j$.
By construction of the bases we have, for $j = 1, \dots, \ell$, that $\beta_{j, k} \in m x_j S_j$ for $k = 1, \dots, d_j$.
This then implies that $(\beta_{1, k_1}, \dots, \beta_{j, k_j}) \in mS$ for $k_j = 1, \dots, d_j$ and $j = 1, \dots, \ell$,
    and thus that $p_{1, \beta_{1, k_1}} \dots p_{\ell, \beta_{\ell, k_\ell}} \in \{p_\alpha\,; \alpha \in \mathbb{R}_+S \cap \mathbb{N}^n\}$.
We also have that
\begin{equation*}
    z^\alpha
    =
    \sum_{k_1 = 1}^{d_1} \dots \sum_{k_j = 1}^{d_j}
        c'_{k_1, \dots, k_\ell} p_{1, \beta_{1, k_1}}(z_1) \dots p_{\ell, \beta_{\ell, k_\ell}}(z_\ell),
\end{equation*}
where $c'_{k_1, \dots, k_\ell} = c_{1, k_1} \dots c_{\ell, k_\ell}$.
So $\{p_\alpha\,; \alpha \in \mathbb{R}_+S \cap \mathbb{N}^n\}$ spans $\mathcal{P}^S(\mathbb{C}^n)$.
We note as well that $\{p_\alpha\,; \alpha \in \mathbb{R}_+S \cap \mathbb{N}^n\}$ is an orthonormal basis for $\mathcal{P}^S(\mathbb{C}^n)$
    as a subspace of $L^2(\Omega_r)$.

We can now write
\begin{equation}
\label{eq:5.8}
    f(z)
    =
    \sum_{\alpha \in mS \cap \mathbb{N}^n}
        c_\alpha p_\alpha(z),
\end{equation}
for $z \in \mathbb{C}^n$, where $c_\alpha = \langle f, p_\alpha \rangle$.
By the Cauchy-Schwarz inequality $|c_\alpha| \leq C_r \|f\|_{L^\infty(\Omega_r)}$,
    where $C_r = \operatorname{Vol}(\Omega_r)^{1/2}$.
Let us now fix $z = (z_1, \dots, z_\ell) \in \mathbb{C}^n$, where $z_j \in \mathbb{C}^{n_j}$, for $j = 1, \dots, \ell$.
If $\alpha = (\alpha_1, \dots, \alpha_\ell)$, where $\alpha_j \in m x_j S_j$ for $j = 1, \dots, \ell$ and $x \in T$, 
    we get, by the submean value property (in each variable), that
\begin{equation*}
    |p_{j, \alpha_j}(a_j)|^2
    \leq
    (\pi r^2)^{-n_j}
        \int_{r\mathbb{D}^{n_j} + a_j} |p_{j, \alpha_j}(\zeta)|^2 d\lambda(\zeta)
    \leq
    (\pi r^2)^{-n_j},
\end{equation*}
for $a_j \in K_j$ and $j = 1, \dots, \ell$.
By the Bernstein-Walsh inequality 
\begin{equation*}
    |p_\alpha(z)|
    \leq
    (\pi r^2)^{-n/2}
        e^{mV^{x_1 S_1}_{K_1}(z_1)}
        \cdots
        e^{mV^{x_\ell S_\ell}_{K_\ell}(z_\ell)}
    =
    \pi^{-n/2} r^{-n}
        e^{m(x_1 V^{S_1}_{K_1}(z_1) + \dots + x_\ell V^{S_\ell}_{K_\ell}(z_\ell))}
\end{equation*}
\begin{equation*}
    \leq
    \pi^{-n/2} r^{-n}
        e^{m\varphi_T(V^{S_1}_{K_1}(z_1), \dots, V^{S_\ell}_{K_\ell}(z_\ell))}.
\end{equation*}
Since $mS \subset m \sigma_S\Sigma \subset [0, m \sigma_S]^n$,
    the number of terms in the sum in equation (\ref{eq:5.8}) is no greater than
    $(m \sigma_S + 1)^n$ where $\sigma_S = \varphi_S(1, \dots, 1)$.
So
\begin{equation}
\label{eq:5.10}
    |f(z)|
    \leq
    C_r \|f\|_{L^\infty(\Omega_r)}
        \pi^{-n/2} r^{-n}
            e^{m\varphi_T(V^{S_1}_{K_1}(z_1), \dots, V^{S_\ell}_{K_\ell}(z_\ell))}
        (m \sigma_S + 1)^n.
\end{equation}
Applying the above inequality on $f^t \in \mathcal{P}^S_{mt}(\mathbb{C}^n)$, for $t \in \mathbb{N}$, we get
\begin{equation*}
    |f(z)|^t
    \leq
    C_r \|f\|_{L^\infty(\Omega_r)}^t
        \pi^{-n/2} r^{-n}
            e^{mt \varphi_T(V^{S_1}_{K_1}(z_1), \dots, V^{S_\ell}_{K_\ell}(z_\ell))}
        (mt \sigma_S + 1)^n.
\end{equation*}
Taking the $t$-th root improves the estimate in (\ref{eq:5.10}) to
\begin{equation*}
    |f(z)|
    \leq
    \|f\|_{L^\infty(\Omega_r)}
            e^{m \varphi_T(V^{S_1}_{K_1}(z_1), \dots, V^{S_\ell}_{K_\ell}(z_\ell))}
        \left (
            C_r
            \pi^{-n/2} r^{-n}
            (mt \sigma_S + 1)^n
        \right )^{1/t}.
\end{equation*}
We can now take the limit as $t$ goes to infinity and then as $r$ goes to zero to get
\begin{equation*}
    |f(z)|
    \leq
        \|f\|_{L^\infty(K)} e^{m \varphi_T(V^{S_1}_{K_1}(z_1), \dots, V^{S_\ell}_{K_\ell}(z_\ell))}.
\end{equation*}
\end{proof}

Theorem \ref{mainresult} now follows from the Siciak-Zakharyuta Theorem \ref{bg:thm.2},
    along with some regularization arguments discussed in the introduction, which we will now restate.
Take $K \subset \mathbb{C}^n$ non-pluripolar and $0 \in S \subset \mathbb{R}_+^n$ compact and convex.
Setting $K_\varepsilon = K + \varepsilon \overline{\mathbb{D}}^n$ we have that $V^{S*}_{K_\varepsilon} \leq 0$ on $K$ and 
    $V^S_{K_\varepsilon} \nearrow V^S_K$ when $\varepsilon \searrow 0$,
    see Propositions $4.8$ and $5.3$ in \cite{MagSigSigSno:2023}.
If $T_j \subset \mathbb{R}_+^n$ is a decreasing sequence of convex compact sets containing $0$ then
    $V^{T_j}_K \searrow V^S_K$, when $j \rightarrow \infty$, if $V^{T_k*}_K \leq 0$ on $K$, for some $k \in \mathbb{N}$,
    see Proposition $4.8$ in \cite{MagSigSigSno:2023}.
\begin{prooftx}{Proof of Theorem \ref{mainresult}}
    Recall that, by (\ref{bg:eq.1}), we have, for $z \in \mathbb{C}^n$,
\begin{equation*}
    \varphi_T(V^{S_1}_{K_1}(z_1), \dots, V^{S_\ell}_{K_\ell}(z_\ell)) \leq V^S_K(z),
\end{equation*}
so the goal is to prove the inverse inequality.
To do this we use Theorem \ref{bg:thm.2} and Proposition \ref{main:prop.1}.
For Theorem \ref{bg:thm.2} to apply we assume some regularity on $S$ and $V^S_K, V^{S_1}_{K_1}, \dots, V^{S_\ell}_{K_\ell}$.
These assumptions are then relaxed using regularization arguments.

To start off we assume that $S = \overline{S \cap \mathbb{Q}^n}$ and 
    $V^S_K, V^{S_1}_{K_1}, \dots, V^{S_\ell}_{K_\ell}$ are all plurisubharmonic.
The second assumption implies that $\varphi_T(V^{S_1}_{K_1}, \dots, V^{S_\ell}_{K_\ell})$ is plurisubharmonic.
Theorem \ref{bg:thm.2} implies that for $z \in \mathbb{C}^{*n}$
\begin{equation*}
    V^S_K(z)
    =
    \varlimsup_{m \rightarrow \infty}
        \sup\{\log|f(z)|^{1/m}\,; f \in \mathcal{P}^S_m(\mathbb{C}^n), f|_K \leq 1\},
\end{equation*}
which, with Proposition \ref{main:prop.1}, yields, for $z \in \mathbb{C}^{*n}$,
\begin{equation*}
    V^S_K(z)
    \leq  
    \varphi_T(V^{S_1}_{K_1}(z_1), \dots, V^{S_\ell}_{K_\ell}(z_\ell)).
\end{equation*}
By assumption $V^S_K$ and $\varphi_T(V^{S_1}_{K_1}(z_1), \dots, V^{S_\ell}_{K_\ell}(z_\ell))$ are plurisubharmonic, so
\begin{equation*}
    V^S_K(z)
    =
    \varlimsup_{\mathbb{C}^{*n} \ni w \rightarrow z}
        V^S_K(w)
    \leq
    \varlimsup_{\mathbb{C}^{*n} \ni w \rightarrow z}
        \varphi_T(V^{S_1}_{K_1}(w_1), \dots, V^{S_\ell}_{K_\ell}(w_\ell))
\end{equation*}
\begin{equation*}
    =
    \varphi_T(V^{S_1}_{K_1}(z_1), \dots, V^{S_\ell}_{K_\ell}(z_\ell)),
\end{equation*}
for $z \in \mathbb{C}^n$, since $\mathbb{C}^n \setminus \mathbb{C}^{*n}$ is pluripolar.
Equality then follows from (\ref{bg:eq.1}).

We now drop the assumptions on $S$ and instead assume that $T \cap \mathbb{R}^{*\ell} \neq \emptyset$.
Let
    $K_{j, \varepsilon} = K_j + \varepsilon \overline{\mathbb{D}}^{n_j}$,
    $K_\varepsilon = K + \varepsilon \overline{\mathbb{D}}^n = K_{1, \varepsilon} \times \dots \times K_{\ell, \varepsilon}$,
    $S_{j, k} = \operatorname{ch}\{(1/k)\Sigma_j \cup S_j\}$, for $k = 1, 2, 3, \dots$, and
\begin{equation*}
    S'_k
    =
    \bigcup_{x \in T} (x_1S_{1, k}) \times \dots \times (x_\ell S_{\ell, k})
    \subset
    \mathbb{R}^n_+.
\end{equation*}
If $x \in T \cap \mathbb{R}^{*\ell}$ then $(x_1 S_{1, k}) \times \dots \times (x_\ell S_{\ell, k}) \subset S'_k$.
Since $x_1S_{1, k}, \dots, x_\ell S_{\ell, k}$ are all convex bodies,
    $S'_k$ is also a convex body.
Consequently, $\overline{S'_k \cap \mathbb{Q}^n} = S'_k$ and
\begin{equation*}
    V^{S'_k}_{K_\varepsilon}(z)
    =
    \varphi_T(
        V^{S_{1, k}}_{K_{1, \varepsilon}}(z_1),
        \dots,
        V^{S_{\ell, k}}_{K_{\ell, \varepsilon}}(z_\ell)
    ).
\end{equation*}
So, by the continuity of $\varphi_T$, we have
\begin{equation*}
    V^S_{K_\varepsilon}(z)
    =
    \lim_{k \rightarrow \infty}
        V^{S'_k}_{K_\varepsilon}(z)
    =
    \lim_{k \rightarrow \infty}
        \varphi_T(
            V^{S_{1, k}}_{K_{1, \varepsilon}}(z_1),
            \dots,
            V^{S_{\ell, k}}_{K_{\ell, \varepsilon}}(z_\ell)
        )
\end{equation*}
\begin{equation*}
    =
    \varphi_T(
        V^{S_1}_{K_{1, \varepsilon}}(z_1),
        \dots,
        V^{S_\ell}_{K_{\ell, \varepsilon}}(z_\ell)
    ).
\end{equation*}
and
\begin{equation*}
    V^S_K(z)
    =
    \lim_{\varepsilon \rightarrow 0}
        V^S_{K_\varepsilon}(z)
    =
    \lim_{\varepsilon \rightarrow 0}
        \varphi_T(
            V^{S_1}_{K_{1, \varepsilon}}(z_1),
            \dots,
            V^{S_\ell}_{K_{\ell, \varepsilon}}(z_\ell)
        )
\end{equation*}
\begin{equation*}
    =
    \varphi_T(
        V^{S_1}_{K_1}(z_1),
        \dots,
        V^{S_\ell}_{K_\ell}(z_\ell)
    ).
\end{equation*}

Lastly we assume $T \cap \mathbb{R}^{*\ell} = \emptyset$ and $T \neq \{0\}$.
If $T = \{0\}$ then $V^S_K = 0$ and $\varphi_T = 0$ so there is nothing to prove.
By rearranging the coordinates, we can assume that $T = A \times \{0\}$,
    where $A \subset \mathbb{R}_+^k$ satisfies that $A \cap \mathbb{R}^{*k} \neq \emptyset$ and $k < \ell$.
Note that $\varphi_T(\xi) = \varphi_A(\xi_1, \dots, \xi_k)$, so
    $\varphi_S(\xi) = \varphi_{A'}(\xi_1, \dots, \xi_\nu)$, where $\nu = n_1 + \dots + n_k$ and
\begin{equation*}
    A'
    =
    \bigcup_{x \in A} (x_1 S_1) \times \dots \times (x_k S_k).
\end{equation*}
We then get, by the Liouville theorem for subharmonic functions, that functions in $\mathcal{L}^S(\mathbb{C}^n)$
    only depend on their first $\nu$ variables.
So
\begin{equation*}
    V^S_K(z)
    =
    V^{A'}_{\tilde{K}}(z)
    =
    \varphi_A(V^{S_1}_{K_1}(z_1), \dots, V^{S_k}_{K_k}(z_k))
    =
    \varphi_T(V^{S_1}_{K_1}(z_1), \dots, V^{S_\ell}_{K_\ell}(z_\ell)),
\end{equation*}
for $z \in \mathbb{C}^n$, where $\tilde{K} = K_1 \times \dots \times K_k$.
\end{prooftx}

\section{The weighted case}
We would like to prove a version of Theorem \ref{mainresult} which includes weights.
One approach is to try to find the correct representation of $q$ such that
\begin{equation*}
    V^S_{K, q}(z)
    =
    \varphi_T(V^{S_1}_{K_1, q_1}(z_1), \dots, V^{S_\ell}_{K_\ell, q_\ell}(z_\ell))
\end{equation*}
holds.
A natural first guess is to take $q = \varphi_T(q_1, \dots, q_\ell)$.
The following results will show that this choice of $q$ is not correct.
We will then show that no choice of $q$ will work in the generality of Theorem \ref{mainresult}, since 
    $\varphi_T(V^{S_1}_{K_1, q_1}(z_1), \dots, V^{S_\ell}_{K_\ell, q_\ell}(z_\ell))$ may fail to be maximal outside of $K$,
    which is a necessary condition for $V^S_{K, q}$, for any choice of an admissible weight $q$.
See \cite{MagSigSigSno:2023}, Theorem $6.1$.

\begin{proposition}
\label{weight:prop.1}
Let $n_1, \dots, n_\ell$ be natural numbers,
    $K_j \subset \mathbb{C}^{n_j}$ be compact and non-pluripolar for $j = 1, \dots, \ell$, and
    $T \subset \mathbb{R}_+^\ell$ be compact convex and containing more than one point, and
    $\{0\} \neq S_j \subset \mathbb{R}_+^{n_j}$ be compact convex and containing $0$, for $j = 1, \dots, \ell$,
        such that $V^{S_j*}_{K_j} = 0$, on $K_j$, for $j = 1, \dots, \ell$.
Then there exist admissible weights $q_1, \dots, q_\ell$ on $K_1, \dots, K_\ell$ respectively, such that, for some $z \in \mathbb{C}^n$, 
\begin{equation*}
    V^S_{K, q}(z)
    >
    \varphi_T(V^{S_1}_{K_1, q_1}(z_1), \dots, V^{S_\ell}_{K_\ell, q_\ell}(z_\ell)),
\end{equation*}
where
    $K = K_1 \times \dots \times K_\ell$,
    $q = \varphi_T(q_1, \dots, q_\ell)$, and 
    $S$ is given by equation (\ref{eq:2.7}).
Furthermore, if $T$ is a convex body then the previous statement holds for all constant weights $q_j < 0$, $j = 1, \dots, \ell$.
\end{proposition}
\begin{proof}
Recall that if $\eta \in \mathbb{R}^\ell$ with $|\eta| = 1$ then
    $\{x \in \mathbb{R}^\ell\,; \langle x, \eta \rangle = \varphi_T(\eta)\}$
    and
    $\{x \in \mathbb{R}^\ell\,; \langle x, \eta \rangle = -\varphi_T(-\eta)\}$
    are supporting hyperplanes of $S$ with outward normals $\eta$ and $-\eta$, respectively.
The distance between these parallel hyperplanes is $\varphi_T(\eta) + \varphi_T(-\eta)$,
    since $0 = \varphi_T(0) \leq \varphi_T(\eta) + \varphi_T(-\eta)$.
So $\operatorname{diam} T \geq \varphi_T(\eta) + \varphi_T(-\eta)$,
    where $\operatorname{diam}T$ denotes the diameter of $T$.
If $L \subset T$ is a line segment parallel to $\eta_L$, with $|\eta_L| = 1$,
    then $\varphi_T(\eta_L) + \varphi_T(-\eta_L) \geq \varphi_L(\eta_L) + \varphi_L(-\eta_L) = \operatorname{diam} L$.
We can take $L$ such that $\operatorname{diam} L = \operatorname{diam} T$, so
\begin{equation*}
    \operatorname{diam} T
    = 
    \sup_{|\eta| = 1} (\varphi_T(\eta) + \varphi_T(-\eta)).
\end{equation*}
Consequently, since $T$ contains more than one point, we can find $\eta \in \mathbb{R}^\ell_+$ such that $\varphi_T(\eta) + \varphi_T(-\eta) > 0$.

Since $V^{S_j*}_{K_j} = 0$ on $K_j$,
    Proposition $5.4$ in \cite{MagSigSigSno:2023} implies that $V^{S_j}_{K_j}$ is continuous on $\mathbb{C}^{*n}$ and
    consequently $V^{S_j}_{K_j}(\mathbb{C}^{n_j}) = \mathbb{R}_+$.

Now let 
    $\eta \in \mathbb{R}_+^\ell$ such that $\varphi_T(\eta) + \varphi_T(-\eta) > 0$,
    $z_j \in \mathbb{C}^{n_j}$ be such that $V^{S_j}_{K_j}(z_j) = \eta_j$, and
    $q_j(w) = -\eta_j$, for $w \in K_j$ and $j = 1, \dots, \ell$.
Then $V^{S_j}_{K_j, q_j}(z_j) = 0$ and
\begin{equation*}
    \varphi_T(V^{S_1}_{K_1, q_1}(z_1),
        \dots,
        V^{S_\ell}_{K_\ell, q_\ell}(z_\ell))
    =
    0.
\end{equation*}
But, we have by Theorem \ref{mainresult} that
\begin{equation*}
    V^S_{K, q}(z)
    =
    V^S_K(z)
    +
    \varphi_T(q_1, \dots, q_\ell)
    =
    \varphi_T(V^{S_1}_{K_1}(z_1), \dots, V^{S_\ell}_{K_\ell}(z_\ell))
    +
    \varphi_T(-\eta)
\end{equation*}
\begin{equation*}
    =
    \varphi_T(\eta)
    +
    \varphi_T(-\eta)
    >
    0.
\end{equation*}
So $V^S_{K, q}(z) > \varphi_T(V^{S_1}_{K_1, q_1}(z_1), \dots, V^{S_\ell}_{K_\ell, q_\ell}(z_\ell))$.

Now assume that $T$ is a convex body.
Then $T$ contains a Euclidean ball with radius $r$ and thus
\begin{equation*}
    \varphi_T(\eta) + \varphi_T(-\eta)
    \geq
    2r|\eta|,
\end{equation*}
for all $\eta \in \mathbb{R}^{\ell}$.
So we set $\eta_j = -q_j$ and proceed as previously in the proof.
\end{proof}

Note that all the weights in the previous theorem are negative.
Counterexamples with positive weights can, however, be inferred.
Take $T = \Sigma$.
Then $\varphi_T(\xi) = \max\{\xi_1, \dots, \xi_\ell\}$.
For $a \in \mathbb{R}$ we have that $V^S_{K, q} + a = V^S_{K, q + a}$ and
\begin{equation*}
    \varphi_T(V^{S_1}_{K_1, q_1}, \dots, V^{S_\ell}_{K_\ell, q_\ell}) + a
    =
    \varphi_T(V^{S_1}_{K_1, q_1 + a}, \dots, V^{S_\ell}_{K_\ell, q_\ell + a}).
\end{equation*}

Now we turn to the question of whether there is any way to choose $q$ to attain an equality.
Recall that,
    for compact $K \subset \mathbb{C}^n$,
    compact convex $S \subset \mathbb{R}_+^n$ such that $0 \in S$ and $a \Sigma \subset S$ for some $a > 0$, and
    admissible weight $q$ on $K$,
    Proposition \ref{bg:prop.1} states that the only function $v \in \mathcal{L}^S_+(\mathbb{C}^n)$
        that is maximal on $\mathbb{C}^n \setminus K$ and agrees
    with $V^{S*}_{K, q}$ on $K$, is $V^{S*}_{K, q}$ itself.
This enables us to show that $\varphi_S(V^{S_1}_{K_1, q_1}(z_1), \dots, V^{S_\ell}_{K_\ell, q_\ell}(z_\ell))$ is not generally maximal outside of
    $K_1 \times \dots \times K_\ell$.
\begin{proposition}
\label{weight:prop.2}
Let
    $n_1, \dots, n_\ell$ be natural numbers,
    $K_j \subset \mathbb{C}^{n_j}$ be compact and non-pluripolar for $j = 1, \dots, \ell$, and
    $T \subset \mathbb{R}_+^\ell$ be compact convex body,
    $S_j \subset \mathbb{R}_+^{n_j}$ be compact convex and containing $0$ with $a \Sigma_j \subset S_j$ for some $a$, for $j = 1, \dots, \ell$,
        such that $V^{S_j*}_{K_j} = 0$ on $K_j$, for $j = 1, \dots, \ell$, and
    $q_j < 0$ is a constant weight $K_j$, for $j = 1, \dots, \ell$.
Then
\begin{equation*}
    V(z)
    =
    \varphi_T(V^{S_1}_{K_1, q_1}(z_1), \dots, V^{S_\ell}_{K_\ell, q_\ell}(z_\ell))
    =
    \varphi_T(V^{S_1}_{K_1}(z_1) + q_1, \dots, V^{S_\ell}_{K_\ell}(z_\ell) + q_\ell)
\end{equation*}
    is not maximal on $\mathbb{C}^n \setminus K$.
\end{proposition}
\begin{proof}
Let $K = K_1 \times \dots \times K_\ell$ and $S$ be given by (\ref{eq:2.7}).
We have that $V^{S_j}_{K_j, q_j} \in \mathcal{L}^{S_j}_+(\mathbb{C}^{n_j})$, see Proposition $4.5$ in \cite{MagSigSigSno:2023},
    so there exists a constant $c$ such that $H_{S_j} - c \leq V^{S_j}_{K_j, q_j}$.
Since $\varphi_T(\xi) \leq \varphi_T(\xi - \eta) + \varphi_T(\eta)$ holds for all $\xi, \eta \in \mathbb{R}^n$, we have
\begin{equation*}
    H_S - \varphi_T(c, \dots, c)
    \leq
    \varphi_T(H_{S_1} - c, \dots, H_{S_\ell} - c)
    \leq
    V.
\end{equation*}
So $V|_K = \varphi_T(q_1, \dots, q_\ell)$ and $V \in \mathcal{L}_+^S(\mathbb{C}^n)$.
If $V$ were maximal outside of $K$ then, by Proposition \ref{bg:prop.1}, we would have $V = V^S_{K, q}$ where $q = \varphi_T(q_1, \dots, q_\ell)$,
    contradicting Proposition \ref{weight:prop.1}.
So $V$ is not maximal outside of $K$.
\end{proof}

\section{Convexity of sublevel sets}
Theorem $1.3$ in \cite{NguTan:2021} states that for
    $t > 0$,
    convex body $0 \in S \subset \mathbb{R}^n_+$, and
    compact convex $K \subset \mathbb{C}^n$,
    the sublevel set $\{z \in \mathbb{C}^n \,; V^S_K(z) < t\}$ is convex.
This can not hold in the stated generality, as discussed in Example $9.3$ of \cite{MagSigSigSno:2023}.
The error in the proof in \cite{NguTan:2021} is in the first equality on page $516$,
    where $\max\{a, b\} + \max\{c, d\} = \max\{a + c, b + d\}$ is used.
This identity does not hold generally.
To describe for which sets $S$ the result can not hold we define, for $x \in \mathbb{R}^{*n}_+$, the \emph{simplex given by $x$} as
\begin{equation*}
    \Sigma_x
    =
    \operatorname{ch}\{0, x_1 e_1, \dots, x_n e_n\}
    =
    \{\xi \in \mathbb{R}^n_+ \,; \xi_1/x_1 + \dots + \xi_n/x_n \leq 1\},
\end{equation*}
where $e_1, \dots, e_n$ is the standard basis for $\mathbb{R}^n$.
\begin{proposition}
    Let $S \neq \{0\}$ be a compact convex subset of $\mathbb{R}^n_+$ containing $0$ that is not a simplex.
    Then there exists $t_0 > 0$ such that the sublevel set
\begin{equation*}
    \{z \in \mathbb{C}^n \,; H_S(z) \leq t\}
\end{equation*}
is not convex for all $t > t_0$. 
\end{proposition}
\begin{proof}
Assume first that $S$ contains a neighborhood of $0$ in $\mathbb{R}^n_+$ and
    define $x \in \mathbb{R}^n_+$ by $x_j = \max\{t \in \mathbb{R} \,; te_j \in S\} > 0$.
By assumption there is an $s \in S$ such that $s \not \in \Sigma_x$, and consequently $s_1/x_1 + \dots + s_n/x_n > 1$.
For $a > 1$ we have by Proposition $3.3$ in \cite{MagSigSigSno:2023} that $H_S(a^{1/x_j}e_j) = \log a$, for $j = 1, \dots, n$, and
\begin{equation*}
    H_S(a^{1/x_1}/n, \dots, a^{1/x_n}/n)
    \geq
    s_1\log (a^{1/x_1}/n) + \dots + s_n\log (a^{1/x_n}/n)
\end{equation*}
\begin{equation*}
    =
    (s_1/x_1 + \dots + s_n/x_n) \log a - (s_1 + \dots + s_n) \log n.
\end{equation*}
Setting
\begin{equation*}
    t_0 = \frac{(s_1 + \dots + s_n)\log n}{s_1/x_1 + \dots + s_n/x_n - 1} > 0,
\end{equation*}
$t > t_0$, and $a = e^t$, we have that $\log a > t_0$ and consequently
\begin{equation*}
    H_S(a^{1/x_1}/n, \dots, a^{1/x_n}/n)
    >
    \log a
    =
    H_S(a^{1/x_j}e_j)
\end{equation*}
for $j = 1, \dots, n$.
Since $t = \log a$ we have that $a^{1/x_1}e_j, \dots, a^{1/x_n}e_n$ are all in the sublevel set $\{z \in \mathbb{C}^n \,; H_S(z) \leq t\}$,
    but their average is not.
So the sublevel set is not convex.

Now assume $S$ does not contain a neighborhood of $0$.
Then, by possibly rearranging the variables, we may assume that $H_S(\zeta, 0, \dots, 0) = 0$ for $\zeta \in \mathbb{C}$.
Since $S \neq \{0\}$ we may assume that there is an $s \in S$ such that $s_1 > 0$,
    since otherwise we could write $S = \{0\} \times T$ reducing the problem to a lower dimension.
Let us now fix $t > 0$ and note that
\begin{equation*}
    H_S(\zeta/2, 1/2, \dots, 1/2)
    \geq
    s_1\log |\zeta| - (s_1 + s_2 + \dots + s_n) \log 2,
\end{equation*}
    for $\zeta \in \mathbb{C}$, so we can choose $\tau \in \mathbb{C}$ such that $H_S(\tau/2, 1/2, \dots, 1/2) > t$.
We have that $H_S(\tau, 0, \dots, 0) = H_S(0, 1, \dots, 1) = 0$,
    so both $(\tau, 0, \dots, 0)$ and $H_S(0, 1, \dots, 1)$ are in the sublevel set $\{z \in \mathbb{C}^n \,; H_S(z) \leq t\}$ but their average is not.
So the sublevel set is not convex.
\end{proof}

{\small 
\bibliographystyle{siam}
\bibliography{rs_bibref}

\smallskip\noindent
Science Institute\\
University of Iceland\\
IS-107 Reykjav\'ik\\
ICELAND 

\smallskip\noindent
bergur@hi.is.
}

\end{document}